\documentclass[twoside,12pt,leqno]{amsproc}
\usepackage{amssymb,latexsym,enumerate,stmaryrd}

\usepackage{booktabs}  
\usepackage[pagebackref]{hyperref}
\usepackage{amsrefs}    
\usepackage{etoolbox}   
\hypersetup{citecolor=red, linkcolor=blue, colorlinks=true}
\numberwithin{table}{section}
\numberwithin{figure}{section}

\theoremstyle{plain}
\newtheorem{theorem}{Theorem}[section]
\newtheorem{corollary}[theorem]{Corollary}

\newtheorem{proposition}[theorem]{Proposition}
\newtheorem{lemma}[theorem]{Lemma} 
\theoremstyle{definition}\newtheorem{definition}[theorem]{Definition}
\theoremstyle{definition}\newtheorem{example}[theorem]{Example} 
\theoremstyle{definition} 
\theoremstyle{definition}\newtheorem{remark}[theorem]{Remark}   
\theoremstyle{definition}\newtheorem{problem}[theorem]{Problem} 
\AtEndEnvironment{example}{\null\hfill$\triangle$}

\renewcommand{\geq}{\geqslant}
\renewcommand{\leq}{\leqslant}
\renewcommand{\ge}{\geqslant}
\renewcommand{\le}{\leqslant}

\newcommand{\lhdeq}{\trianglelefteqslant}    

\newcommand{\Alt}{\textrm{Alt}}

\newcommand{\Aut}{\mathrm{Aut}}
\newcommand{\eps}{\varepsilon}

\newcommand{\Inn}{\mathrm{Inn}}
\newcommand{\Out}{\mathrm{Out}}

\newcommand{\sym}{\textup{\textsf{S}}}
\newcommand{\Sym}{\mathrm{Sym}}

\newcommand{\AGL}{\mathrm{AGL}}
\newcommand{\GL}{\mathrm{GL}}

\newcommand{\PSL}{\mathrm{PSL}}
\newcommand{\SL}{\mathrm{SL}}
\newcommand{\soc}{{\mathrm{soc}}}

\newcommand{\C}{\textup{\textsf{C}}}

\newcommand{\HA}{\mathsf{HA}}
\newcommand{\HS}{\mathsf{HS}}
\newcommand{\HC}{\mathsf{HC}}
\newcommand{\TW}{\mathsf{TW}}
\newcommand{\AS}{\mathsf{AS}}
\newcommand{\SD}{\mathsf{SD}}
\newcommand{\CD}{\mathsf{CD}}
\newcommand{\PA}{\mathsf{PA}}

\makeatletter        
\def\@adminfootnotes{%
  \let\@makefnmark\relax  \let\@thefnmark\relax
  \ifx\@empty\@date\else \@footnotetext{\@setdate}\fi
  \ifx\@empty\@subjclass\else \@footnotetext{\@setsubjclass}\fi
  \ifx\@empty\@keywords\else \@footnotetext{\@setkeywords}\fi
  \ifx\@empty\thankses\else \@footnotetext{%
    \def\par{\let\par\@par}\@setthanks}%
  \fi}\makeatother   

\begin{document}

\hyphenation{}

\title{Arc-transitive digraphs with\\ quasiprimitive local actions}
\author[M. Giudici, S.~P.~Glasby, C. H. Li, G. Verret]{Michael Giudici, S.~P.~Glasby, Cai Heng Li, Gabriel Verret}

 \address{Michael Giudici, S.~P.~Glasby$^*$ and Gabriel Verret$^\dag$
\newline\indent Centre for the Mathematics of Symmetry and Computation, 
\newline\indent The University of Western Australia, 
\newline\indent 35 Stirling Highway, Crawley, WA 6009, Australia.}

\address{Cai Heng Li, Department of Mathematics, 
   South University of Science and\newline\indent Technology of China,
   Shenzhen, Guangdong 518055, P. R. China.}

\address{$^*$Also affiliated with The Department of Mathematics, 
\newline\indent University of Canberra, ACT 2601, Australia.}

\address{$\dag$Also affiliated with FAMNIT, University of Primorska, 
\newline\indent Glagolja\v{s}ka 8, SI-6000 Koper, Slovenia.
\newline\indent Current address: Department of Mathematics, The University of Auckland,
\newline\indent Private Bag 92019, Auckland 1142, New Zealand.}

\email{Michael.Giudici@uwa.edu.au}
\email{Stephen.Glasby@uwa.edu.au}
\email{lich@sustc.edu.cn}
\email{g.verret@auckland.ac.nz}

\thanks{This research was supported by Australian Research Council grants DE130101001 and DP150101066.}



\date{\today}

\begin{abstract}
Let $\Gamma$ be a finite $G$-vertex-transitive digraph. The \emph{in-local action} of $(\Gamma,G)$ is the permutation group $L_-$ induced by the vertex-stabiliser on  the set of in-neighbours of $v$. The \emph{out-local action} $L_+$ is defined analogously. Note that $L_-$ and $L_+$ may not be isomorphic. We thus consider the problem of determining which pairs $(L_-,L_+)$ are possible. We prove some general results, but pay special attention to the case when $L_-$ and $L_+$ are both quasiprimitive. (Recall that a permutation group is \emph{quasiprimitive} if each of its nontrivial normal  subgroups is transitive.) Along the way, we prove a structural result about pairs of finite quasiprimitive groups of the same degree, one being (abstractly) isomorphic to a proper quotient of the other.
\end{abstract}

\maketitle

\section{Introduction}\label{sec:intro}

A \emph{digraph} $\Gamma$ is a pair $(V,A)$ where $A$ is a binary relation on~$V$. The set $V$ is called the \emph{vertex set} of $\Gamma$ and its elements are  \emph{vertices}, while $A$ is the \emph{arc set} and its elements are \emph{arcs}. If $(u,v)\in A$, then $u$ is an \emph{in-neighbour} of $v$, and $v$ is an \emph{out-neighbour} of $u$. The set of in-neighbours of a vertex $v$ of $\Gamma$ is denoted $\Gamma^-(v)$ and the set of out-neighbours is denoted $\Gamma^+(v)$.

An \emph{automorphism} of $\Gamma$ is a permutation of~$V$ that preserves $A$. We say that $\Gamma$ is \emph{$G$-vertex-transitive}, or \emph{$G$-arc-transitive}, if $G$ is a group of automorphisms of $\Gamma$ that is transitive on $V$, or $A$, respectively.

If $\Gamma$ is $G$-vertex-transitive, then the \emph{in-local action}  of $(\Gamma,G)$ is the permutation group induced by the vertex-stabiliser $G_v$ on  $\Gamma^-(v)$. The \emph{out-local action} is defined analogously. As we will see in Section~\ref{S:HNN}, the in- and out-local actions need not be permutation isomorphic. This leads to the following natural definition and problem.

We say that two permutation groups  $L_+$ and $L_-$ are \emph{compatible} if there exists a finite $G$-vertex-transitive digraph $\Gamma$ such that $(\Gamma,G)$ has in-local action $L_-$ and out-local action $L_+$.

\begin{problem}\label{MainProblem}
Determine which permutation groups are compatible.
\end{problem}

It appears to be quite difficult to solve Problem~\ref{MainProblem} in full generality. One case which is particularly interesting is when the permutation groups are transitive. (By Lemma~\ref{lem:prop}, if one is transitive, so is the other.) Much of the literature on this topic uses the language of orbitals. Let $G$ be a transitive permutation group on a set $V$, let $u,v\in V$ and let $A$ be the orbit of $(u,v)$ under $G$. The set $A$ is an \emph{orbital} of $G$ and the digraph $\Gamma$ with vertex set $V$ and arc set $A$ is called an \emph{orbital digraph}. Clearly, $\Gamma$ is both $G$-vertex-transitive and $G$-arc-transitive. The in- and out-local actions of $(\Gamma,G)$ are called \emph{paired subconstituents of $G$ (with respect to $A$)}, see \cite{Cambook}*{p.~60}. Using this terminology, the transitive case of Problem~\ref{MainProblem} can be rephrased in the following way: determine which pairs of permutation groups can arise as  paired subconstituents of a finite permutation group.

An unpublished result of Sims shows that nontrivial compatible transitive groups have a common nontrivial homomorphic image~\cite{Quirin}*{Lemma~4.1}.  Cameron showed that, if a transitive group $L$ has a so called ``suitable'' property, then so does every group compatible with $L$~\cite{Cameron}*{Section 1}. Knapp gave a variant of this result by showing that, if a transitive group $L$ has a ``well suitable'' property (which is a stronger requirement than having a suitable property), then every group compatible with $L$ is permutation isomorphic to $L$. Knapp also showed that well suitable properties include: $2$-transitivity, $2$-homogeneity, and having prime degree~\cite{Knapp}*{Section 3}. Finally, Knapp proved that, given two compatible quasiprimitive groups, one is a homomorphic image of the other~\cite{Knapp}*{Theorem~3.3}.

In the first half of the paper, we approach Problem~\ref{MainProblem} in some generality. We first show that compatible groups must have some basic properties in common, including the degree and the number of orbits (Lemma~\ref{lem:prop}). In Section~\ref{S:HNN}, we give a purely group-theoretic characterisation of transitive compatible groups (Theorem~\ref{theo:ProduceExamples}). We then use this to construct various interesting examples of transitive compatible groups.

In the second half of the paper, we consider Problem~\ref{MainProblem} in the case where both groups are quasiprimitive. (Recall that a permutation group is \emph{quasiprimitive} if each of its nontrivial normal  subgroups is transitive. We follow the partition of finite quasiprimitive groups into $8$ types given by~\cite{P}, see Section~\ref{sec:types}.) Our main result is the following.

\begin{theorem}\label{theo:main2}If $L_-$ and $L_+$ are compatible quasiprimitive groups, then either $L_-$ is (abstractly) isomorphic to~$L_+$, or, interchanging $L_-$ and $L_+$ if necessary, $L_+$ is a proper quotient of  $L_-$, and the pair $(L_-, L_+)$ has type $(\HS,\AS)$ or~$(\HC,\TW)$.
\end{theorem}

As an immediate corollary, if two compatible quasiprimitive groups are not isomorphic, then their degree must be the power of the order of a finite nonabelian simple group. (As this is the only possibility for the degree of a group of type $\HS$ or $\HC$.)

In the process of proving Theorem~\ref{theo:main2}, we also prove the following result, which we believe is of independent interest.

\begin{theorem}\label{theo:main}
Let $G$ and $H$ be finite quasiprimitive groups of the same degree. If $H$ is (abstractly) isomorphic to a proper quotient of $G$, then the type of the pair $(G,H)$ is one of the following:  $(\HS,\AS)$, $(\HC,\TW)$, $(\HA,\AS)$ or $(\HA,\PA)$.
\end{theorem}

Finally, we note that all the exceptional pairs of types which occur in Theorems~\ref{theo:main2} and~\ref{theo:main} cannot be avoided (see Examples~\ref{ExampleHCTW} and~\ref{Ex:HAASPA}).

\section{Common properties of compatible groups}\label{sec:prop}

We say that $H$ is a \emph{section} of $G$ if $H$ is isomorphic to a quotient of a subgroup of $G$. If, in addition, $H$ is simple, then we say that it is  a \emph{simple section} of $G$.

\begin{lemma}\label{lem:cf}
If $1=G_n\lhdeq G_{n-1}\lhdeq\cdots\lhdeq G_1\lhdeq G_0=G$ is a subnormal series for $G$ and $T$ is a simple section of $G$, then $T$ is a simple section of $G_i/G_{i+1}$ for some $i\in \{0,\ldots, n-1\}$.
\end{lemma}
\begin{proof}
Let $i$ be largest with respect to the property that $T$ is a section of $G_i$. Note that $i\in \{0,\ldots, n-1\}$. By definition, there exist $R$ and $S$ such that $S\lhdeq R\leq G_i$ and $R/S\cong T$.  Now, $S(R\cap G_{i+1})$ is a normal subgroup of $R$ containing $S$. Since $R/S$ is simple, either $S(R\cap G_{i+1})=R$ or $R\cap G_{i+1}\leq S$. In the former case, 
$$T\cong R/S=S(R\cap G_{i+1})/S\cong (R\cap G_{i+1})/(S\cap G_{i+1})$$
and $T$ is a section of $G_{i+1}$, contradicting the maximality of $i$. In the latter case, $R/S$ is a section of $R/(R\cap G_{i+1})$ but
$$R/(R\cap G_{i+1})\cong RG_{i+1}/G_{i+1}\leq G_i/G_{i+1}$$
and thus $T$ is a section of $G_i/G_{i+1}$. 
\end{proof}

The main result of this section is the following, which shows that compatible permutation groups share certain properties. 

\begin{lemma}\label{lem:prop}
Compatible groups have the same degree, the same number of orbits, and the same simple sections. In particular, their orders have the same prime divisors and, either both are soluble, or neither is.
\end{lemma}
\begin{proof}
Let $L_-$ and $L_+$ be compatible permutation groups. By definition, there exists a finite $G$-vertex-transitive digraph $\Gamma$ such that $(\Gamma,G)$ has in-local action $L_-$ and out-local action $L_+$. Recall that, for $\eps\in\{\pm\}$, $L_\eps$ is the permutation group induced by a vertex-stabiliser $G_v$ on $\Gamma^\eps(v)$.

Each arc of $\Gamma$ contributes 1 to the sum $\sum_{v\in V} |\Gamma^-(v)|$, and $1$ to the sum $\sum_{v\in V}|\Gamma^+(v)|$. These sums are thus equal. Since $\Gamma$ is $G$-vertex-transitive, the summands are constant and, since $V$ is finite,  we deduce that $|\Gamma^+(v)|=|\Gamma^-(v)|$ for all $v\in V$. This shows that $L_-$ and $L_+$ have the same degree.

Since $\Gamma$ is $G$-vertex-transitive, the number of orbits of $L_-$ and $L_+$ are both equal to the number of orbits of $G$ on the arc set of $\Gamma$.

Let $\Gamma^*$ be a connected component of $\Gamma$ and let $G^*$ be the permutation group induced on the set of vertices of $\Gamma^*$ by the stabiliser in $G$ of this set of vertices. It is not hard to see that $(\Gamma^*,G^*)$ has the same in- and out-local action as $(\Gamma,G)$. We may thus assume that $\Gamma^*$ is connected, and even strongly connected, by~\cite{GR}*{Lemma~2.6.1}.

Let $(v_1, v_2,\ldots, v_n)$ be an ordering of the vertices of $\Gamma$ such that, for each $i\in\{2,\ldots,n\}$, there exists some $j$ with $j<i$ such that $v_i$ is an out-neighbour of~$v_j$. (Such an ordering always exists since $\Gamma$ is strongly connected.) For $i\in \{1,\dots,n\}$, let $G_i$ be the subgroup of $G$ fixing $(v_1,\dots,v_i)$ and all their out-neighbours pointwise. If $i\in\{2,\dots, n\}$, then there exists $j<i$ such that $v_i$ is an out-neighbour of $v_j$, hence $v_i$ is fixed by $G_{i-1}$. By definition, $G_i$ is the kernel of the action of $G_{i-1}$ on $\Gamma^+(v_i)$. It follows that $G_i$ is normal in $G_{i-1}$ with $G_{i-1}/G_i$ isomorphic to a subgroup of $L_+$. Thus  $1=G_n\lhdeq G_{n-1}\lhdeq\cdots\lhdeq G_1\lhdeq G_{v_1}$ is a subnormal series for $G_{v_1}$ with all factor groups isomorphic to subgroups of $L_+$. By Lemma~\ref{lem:cf},  every simple section of $G_{v_1}$ is a section of $L_+$. On the other hand, $G_{v_1}/G_1\cong L_+$ and thus $L_+$ and $G_{v_1}$ have the same simple sections.

Repeating the above argument with in-neighbours instead of out-neighbours, we get that $L_-$ and $G_{v_1}$ have the same simple sections and thus so do $L_+$ and $L_-$.  This concludes the proof.
\end{proof}

\section{A characterisation of compatible transitive groups and some examples}\label{S:HNN}

In this section, we consider compatible transitive groups. Our main result is the following.

\begin{theorem}\label{theo:ProduceExamples}
If $L_-$ and $L_+$ are transitive permutation groups, then they are compatible if and only if there exists a finite group $H$ with two isomorphic subgroups $H_-$ and $H_+$ such that, for each $\eps\in\{\pm\}$, the permutation group induced by the action of $H$ on the set of cosets of $H_\eps$ is permutation isomorphic to $L_\eps$.
\end{theorem}
\begin{proof}
First, suppose that $L_-$ and $L_+$ are compatible. By definition,  there exists a finite $G$-vertex-transitive digraph $\Gamma$ such that $(\Gamma,G)$ has in-local action $L_-$ and out-local action $L_+$. Let $v$ be a vertex of $\Gamma$, let $v_+$ be an out-neighbour of $v$, let $v_-$ be an in-neighbour of $v$, let $H=G_v$, let $H_+=G_{(v,v_+)}$, and let $H_-=G_{(v,v_-)}$. For $\eps\in\{\pm\}$, $L_{\eps}$ is the permutation group induced by $H$ on  $\Gamma^\eps(v)$. This is a transitive permutation group with point-stabiliser $H_\eps$ and thus $L_{\eps}$ is permutation isomorphic to the permutation group induced by the action of $H$ on the set of cosets of $H_\eps$. Finally, since $L_+$ is transitive, $\Gamma$ is $G$-arc-transitive and thus there exists $g\in G$ such that $(v_-,v)^g=(v,v_+)$. It follows that $(H_-)^g=H_+$ hence $H_-\cong H+$.

Conversely, suppose that there exists a finite group $H$ with two isomorphic subgroups $H_-$ and $H_+$ such that, for each $\eps\in\{\pm\}$, the permutation group induced by the action of $H$ on the cosets of $H_\eps$ is permutation isomorphic to $L_\eps$. Let $S$ be the symmetric group on $H\times\C_2$.  By~\cite[Lemma 2.3]{Glaub}, there exists an embedding of $H$ into $S$ and an element $g$ of $S$ such that $(H_-)^g=H_+=H\cap H^g$. 

Let $G=\langle H,g\rangle$, let $V$ be the set of cosets of $H$ in $G$ and let $v=H$, viewed as an element of $V$. Note that $G$ acts transitively (but perhaps not faithfully) on $V$, and that $G_v=H$.  Let $A=(v,v^g)^{G}$ and let $\Gamma$ be the digraph $(V,A)$. By construction, $G$ acts transitively on $A$ and thus $G_v$ acts transitively on $\Gamma^+(v)$. Note that $v^g\in \Gamma^+(v)$ and  $G_{vv^g}=H\cap H^g=H_+$. It follows that $G_v^{\Gamma^+(v)}$ is permutation isomorphic to $L_+$. Similarly, $(v^{g^{-1}},v)\in A$ hence $v^{g^{-1}}\in \Gamma^-(v)$ and $G_{v^{g^{-1}}v}=(G_{vv^g})^{g^{-1}}=(H_+)^{g^{-1}}=H_-$ therefore $G_v^{\Gamma^-(v)}$ is permutation isomorphic to $L_-$. Finally, let $\widehat{G}$ be the permutation group induced by the action of $G$ on $V$. Clearly, $\widehat{G}$ is a group of automorphisms of $\Gamma$, $\Gamma$ is $\widehat{G}$-arc-transitive and, by the preceding argument, $(\Gamma,\widehat{G})$ has in-local action $L_-$ and out-local action $L_+$. This completes the proof.
\end{proof}

\begin{remark}
Note that the proof of Theorem~\ref{theo:ProduceExamples} gives more than required. Indeed, we prove that, given a finite group $H$ with two isomorphic subgroups $H_-$ and $H_+$ and $L_-$ and $L_+$ as in the statement, we can find a $G$-arc-transitive digraph $\Gamma$ witnessing that $L_-$ and $L_+$ are compatible such that $|G|\leq (2|H|)!$.
\end{remark}

Theorem~\ref{theo:ProduceExamples} is very powerful and makes it relatively easy to construct many examples of compatible transitive groups. For example, one of the main results of~\cite{CLP} is the construction of a $G$-arc-transitive digraph $\Gamma$ with  $G_v\cong\Alt(6)$ and $(\Gamma,G)$ having in-local and out-local actions the two inequivalent transitive actions of $\Alt(6)$ of degree $6$. Since, in both of these actions, the point-stabiliser is isomorphic to $\Alt(5)$, the existence of such a digraph follows immediately from Theorem~\ref{theo:ProduceExamples}. (In fact, infinitely many examples are constructed in~\cite{CLP} but, given one example, one can always construct infinitely many others using standard covering techniques.)

In the rest of this section, we use Theorem~\ref{theo:ProduceExamples} to produce a few interesting examples of compatible transitive groups.

By Lemma~\ref{lem:prop}, compatible groups must have the same simple sections.  Our next example shows that they need not have the same composition factors.
\begin{example}
Let $H=\Alt(5)\times\Alt(6)$, and let $H_-=\Alt(5)\times 1$ and $H_+=1\times\Alt(5)$, considered as subgroups of $H$. Note that $H_-\cong H_+$. For $\eps\in\{\pm\}$, let $K_\eps=\textup{core}_H(H_\eps)$ and let $L_\eps=H/K_\eps$. Note that $K_-=\Alt(5)\times 1$ while $K_+=1$, thus $L_-\cong\Alt(6)$ and $L_+\cong\Alt(5)\times\Alt(6)$. By Theorem~\ref{theo:ProduceExamples}, $L_-$ and $L_+$ are compatible.
\end{example}

We next give an example of a non-regular primitive group that is compatible with a regular non-primitive group.

\begin{example}\label{eg:primnotcompat}
Let $T$ be a nonabelian simple group, let  $H=T\times T$, let $H_-=T\times 1\leq H$ and let $H_+$ be the diagonal subgroup of $H$. Note that $H_-$ and $H_+$ are both isomorphic to $T$, $H_-$ is normal in $H$ and $L_-=H/H_-\cong T$ while $H_+$ is maximal and core-free in $H$, hence the action of $H$ on the cosets of $H_+$ is faithful and primitive. By Theorem~\ref{theo:ProduceExamples}, the regular permutation group $T$ is compatible with the primitive permutation group $T\times T$. Thus, even when $L_-$ and $L_+$ share a common composition factor $T$, the multiplicities of $T$ can vary.
\end{example}

Specialising Theorem~\ref{theo:ProduceExamples} to the case of regular permutation groups yields the following.

\begin{corollary}\label{cor:ProduceExamples2}
If $L_-$ and $L_+$ are regular permutation groups, then they are compatible if and only if there exists a finite group $H$ with two isomorphic normal subgroups $H_-$ and $H_+$ such that $H/H_-\cong L_-$ and $H/H_+\cong L_+$.
\end{corollary}

In view of Corollary~\ref{cor:ProduceExamples2}, we feel that determining  which regular permutation groups are compatible is a very interesting problem, even from a purely group-theoretical standpoint and without the graph-theoretical motivation. This problem does not seem easy in general. As an example of one of the difficulties involved, we show that the compatibility relation is not an equivalence relation, even on regular permutation groups.

\begin{example}
Let $H=\Alt(5)\times\Sym(5)$, $H_-=\Alt(5)\times 1$ and $H_+=1\times\Alt(5)$. Note that $H_-$ and $H_+$ are isomorphic normal subgroups of $H$. By Corollary~\ref{cor:ProduceExamples2}, it follows that, as regular permutation groups, $H/H_-\cong \Sym(5)$ and $H/H_+\cong\Alt(5)\times\C_2$ are compatible.

Similarly, let $H=\SL(2,5)\times\C_2$  and let $H_-=\C_2\times 1$ and $H_+=1\times\C_2$, considered as normal subgroups of $H$. Again, $H_-\cong H_+$  and it follows by Corollary~\ref{cor:ProduceExamples2} that the regular permutation groups $\Alt(5)\times\C_2$ and $\SL(2,5)$ are compatible.

On the other hand, $\SL(2,5)$ and $\Sym(5)$ have no common nontrivial homomorphic image hence they are not compatible by~\cite{Quirin}*{Lemma 4.1}.
\end{example}

We end this section with the following result.

\begin{proposition}\label{coolcoolProp}
Let $L$ be a group and let $(F_1,\ldots,F_n)$ be the set of factor groups for a subnormal series of $L$. Then, as regular permutation groups, $L$ and $F_1\times\cdots\times F_n$ are compatible.
\end{proposition}
\begin{proof}
Choose a subnormal series 
$$1=X_0\trianglelefteqslant X_1\trianglelefteqslant\cdots\trianglelefteqslant X_n=L,$$ 
with $X_i/X_{i-1}\cong F_i$ for $i\in\{1,\ldots,n\}$. Let 
\begin{align*}
H&=X_1\times \cdots\times X_{n-1}\times X_n,\\
H_-&= X_1\times\cdots\times X_{n-1}\times 1\quad\textup{and}\\
H_+&=X_0\times\cdots \times X_{n-2}\times X_{n-1}.
\end{align*}
Note that $H_-$ and $H_+$ are isomorphic subgroups of $H$ (for $H_+$, in the $i$th direct factor, we use the natural embedding of $X_{i-1}$ in $X_i$).  We have $H/H_-\cong X_n=L$ and 
$$H/H_+\cong (X_1/X_0)\times\cdots\times(X_{n-1}/X_{n-2})\times(X_{n}/X_{n-1})\cong F_1\times\cdots \times F_{n-1}\times F_n.$$
The conclusion now follows by Corollary~\ref{cor:ProduceExamples2}.
\end{proof}

Proposition~\ref{coolcoolProp} has a number of interesting corollaries. For example, as regular permutation groups, an elementary abelian $p$-group is compatible with every group of the same order. 

We could say more about the problem of  determining which regular permutation groups are compatible but this would take us too far afield. Instead, we move on to the main topic of this paper, namely compatible quasiprimitive groups. We first need a few preliminaries.

\section{Preliminaries to the proofs of the main theorems}\label{sec:qp}

The following result follows from~\cite{CompFactors}*{Theorem~1} (set $q=p$ and note that $\eps_2=1$ and $\eps_p\le3/2$).

\begin{proposition}\label{compfactors}
Let $p$ be a prime. An irreducible subgroup of $\GL(d,p)$ has at most $d-1$ composition factors of order $p$.
\end{proposition}

(In fact, the current application was our initial motivation for~\cite{CompFactors}.) We will also need the following number-theoretic lemma.

\begin{lemma}\label{L1}
If $k,\ell,x$ are integers greater than $1$, then
\begin{enumerate}[{\rm (a)}]
\item $x^k$ does not divide $k!$, and\label{first}
\item if $\ell$ divides $k$, then $4^{k-k/\ell}$ does not divides $k!$.\label{second}
\end{enumerate}
\end{lemma}
\begin{proof}
Let $p$ be a prime and denote the $p$-part of $k!$ by $(k!)_p$. It is well-known and easy to see that
$$
  \log_p(k!)_p\le\frac{k-1}{p-1}\leq k-1.
$$
This immediately shows~(a). Since $\ell\geq 2$, we have $\log_2(4^{k-k/\ell})=2(k-k/\ell)\geq k>\log_2(k!)_2$ and (b) follows. 
\end{proof}

\subsection{Quasiprimitive groups}\label{sec:types}
 We will use the subdivision of quasiprimitive groups into eight types called $\HA, \HS, \HC, \TW, \AS, \SD, \CD$, and $\PA$ given by Praeger in \cite{P}. We now collect a few facts about these groups that can be found in~\cite{P}.

Let $G$ be a quasiprimitive group on $\Omega$ and let $M=\soc(G)$ be the socle of $G$. We have $M\cong T^k$ for some simple group $T$ and $k\geq 1$.  If $T$ is abelian, then $G$ is said to be of type $\HA$. In this case, $T$ is cyclic of order a prime $p$, $M$ is elementary abelian and regular, and a point-stabiliser in $G$ is an irreducible subgroup of $\GL(k,p)$.

Otherwise, $T$ is nonabelian. In this case, the centraliser of $M$ in $G$ is trivial and thus $G$ embeds in $\Aut(M)\cong\Aut(T)\wr\Sym(k)$.   Moreover, $M$ is the unique minimal normal subgroup of $G$, unless $G$ is of type $\HS$ or $\HC$, in which case $G$ has exactly two minimal normal subgroups and they are isomorphic to $T^{k/2}$. In this case, $T^{k/2}\rtimes \Inn(T^{k/2})\leq G\leq T^{k/2}\rtimes \Aut(T^{k/2})$.

If $k=1$, then $G$ is of type $\AS$. If $G$ has type $\PA$, then it acts faithfully on some partition $\mathcal{P}$ of $\Omega$ that can be identified with a set $\Delta^k$ with $|\Delta|\geq 5$ and $k\geq 2$. In particular, $G$ embeds in $H\wr\sym_k$, where $H$ is a  quasiprimitive group on $\Delta$ of type $\AS$. Moreover, in cases $\AS$ and $\PA$, there is a proper subgroup $R$ of $T$ such that $M_\omega$ is isomorphic to a subgroup of $R^k$. 

We also note that if $G$ has type  $\HA$ or  $\TW$, then $M$ is regular. The only other case where $M$ can be regular is if $G$ has type $\AS$. Finally, we note the degree of $G$ according to its type in Table~\ref{TblONS}.

\begin{center}
\begin{table}[!ht]
\caption{The degree of a quasiprimitive group as a function of its~type.}\label{TblONS}
\begin{tabular}{r|cccccccc}
\toprule
  &$\HA$&$\HS$&$\HC$&$\TW$&$\AS$&$\SD$&$\CD$&$\PA$\\ \hline
Degree &$p^k$&$|T|$&$|T|^{k/2}$&$|T|^k$&$-$&$|T|^{k-1}$&$|T|^{k-k/\ell}$& $mx^k$\\
Constraints& & &$k\ge4$&$k\ge2$& &$k\ge2$&$2\le\ell\mid k$& $m\geq 1$, $x\geq 5$, $k\geq 2$ \\
\bottomrule
\end{tabular}
\end{table}
\end{center}

\begin{definition}\label{def:comp}
For a finite group $G$, we denote by $[G]$ the multiset of composition factors in a composition series for $G$. We write $[A][B]$ for the union $[A]\cup [B]$ of multisets.
\end{definition}

The following result about quasiprimitive groups is used in our proof of Theorem~\ref{theo:main2}.

\begin{lemma}\label{CompFactorsPA}
If $G$ is a quasiprimitive group of type $\AS$ or $\PA$, then $[\soc(G)]\not\subseteq[G_v]$.
\end{lemma}

\begin{proof}
Let $M=\soc(G)\cong T^k$ and suppose that $[\soc(G)]\subseteq [G_v]$, that is $[T^k]\subseteq [G_v]$. We seek a contradiction. As discussed above, in cases $\AS$ and $\PA$, there is a proper subgroup $R$ of $T$ such that $M_v\leqslant R^k$. In particular, $T$ is not a composition factor of $M_v$ and thus $[T^k]\subseteq [G_v/M_v]$. On the other hand, $G_v/M_v\cong G_vM/M\leqslant G/M\lesssim \Out(T)\wr \Sym(k)$.  By the Schreier conjecture, $\Out(T)$ is soluble, and so $|T|^k$ divides $k!$. However, $4$ divides $|T|$,  contradicting Lemma \ref{L1}.
\end{proof}

We also need the following result.

\begin{lemma}\label{lem:outT}
Let $G$ be a quasiprimitive permutation group of type $\AS$ on a set $\Omega$ and with socle~$T$. Then $|\Omega|>|\Out(T)|$.
\end{lemma}
\begin{proof}
Let $\mathcal{P}$ be a maximal $G$-invariant partition of  $\Omega$. Since $G$ is quasiprimitive, it acts faithfully on $\mathcal{P}$ and the maximality of $\mathcal{P}$ implies that this action is primitive. By \cite{GMP}*{Theorem 1.2} we have that $|\mathcal{P}|>|\Out(T)|$ and so $|\Omega|>|\Out(T)|$.
\end{proof}

\section{Proof of Theorem~\ref{theo:main}}\label{S2}

Let $G$ and $H$ be quasiprimitive groups on the same finite set $\Omega$. Assume that $H\cong G/K$ for some $K>1$, that is, $H$ is  isomorphic to a proper quotient of $G$.

\subsection{$G$ has a simple socle}\label{sec:first}

Let $T=\soc(G)$. Note that $T$ is nonabelian  and simple, $T\leq K$ and thus  $G/K$ is isomorphic to a section of $\Out(T)$. By  the Schreier conjecture, $G/K$ is soluble and thus has type $\HA$ and $|\Omega|$ is a prime-power. Furthermore, $|\Omega|$ divides $|G/K|$ and thus $|\Out(T)|$.  This contradicts Lemma \ref{lem:outT}.

\subsection{$G$ has an abelian socle}

In this case, $G$ has prime power degree, say $p^k$. Since the order
of a nonabelian simple group is not a prime power, it
follows from Table~\ref{TblONS} that $G/K$ has type $\HA, \AS$, or
$\PA$. The case when $G/K$ has type $\AS$ or $\PA$ appears in the statement of
Theorem~\ref{theo:main}. It remains to exclude the case when $G/K$ has
type $\HA$.

Note that $G$ has a unique minimal normal subgroup $M$ and it is
elementary abelian and regular. In particular, $M\leq K$. Let $G_v$ be
a point-stabiliser in $G$. By the second isomorphism theorem, $G/K$ is
isomorphic to a quotient of $G/M\cong G_v$, itself a subgroup of
$\GL(k,p)$. Recall that $G_v$ is an irreducible subgroup of
$\GL(k,p)$. By Proposition~\ref{compfactors}, $G_v$ has at most $k-1$
composition factors of order $p$. Since $G/K$ is of type $\HA$ and $|\Omega|=p^k$, its
socle is also elementary abelian of order $p^k$, a contradiction.

\subsection{$G$ has a nonabelian and nonsimple socle}
Let $M=\soc(G)$. Note that $M=T^k$  with $T$ nonabelian simple and $k\geq 2$.

\subsubsection{$M\leq K$}
Let $\overline{\phantom{n}}$ be the natural epimorphism
$G\twoheadrightarrow G/M$. Recall that $$G\lesssim\Aut(M)\cong\Aut(T)
\wr \sym_k$$ and thus $\overline{G}\lesssim\Out(T) \wr
\sym_k=\Out(T)^k \rtimes \sym_k$. We will consider $\overline{G}$ as a
subgroup of $\Out(T)^k \rtimes \sym_k$.  Let
$X=\Out(T)^k\cap\overline{G}$.

Suppose first that $X\le\overline{K}$ and hence $G/K$ is a section of
$\sym_k$. However, $|\Omega|$ divides $|G/K|$ and hence $k!$. If $G$
has type $\PA$ then $|\Omega|=mx^k$ for some $x\geq 5$ and $m\geq 1$,
contradicting Lemma~\ref{L1}(\ref{first}). Similarly, if $G$ has type
$\TW$ then $|T|^k$ divides $k!$, a contradiction. If $G$ has type
$\HC$ or $\HS$, then $k$ is even and $|T|^{k/2}$ divides $k!$. As
$4$ divides $|T|$, this again contradicts Lemma~\ref{L1}(\ref{first}).
If $G$ has type $\CD$, then $|\Omega|=|T|^{k-k/\ell}$ and this can not
divide $k!$ by Lemma~\ref{L1}(\ref{second}). The same argument applied
with $k=\ell$ yields that $G$ cannot have type $\SD$.

We may thus assume that $X\not\le\overline{K}$. Recall that, by the Schreier conjecture, $\Out(T)^k$ is soluble. In particular,
$\overline{G}/\overline{K}$ has a nontrivial soluble normal subgroup,
namely $X\overline{K}/\overline{K}$. Since
$\overline{G}/\overline{K}\cong G/K$, it follows that $G/K$ has an
abelian minimal normal subgroup and thus $G/K$ has type $\HA$ and
degree a power of a prime, say $p$. Note that $p$ divides $|X|$ and
thus also divides $|\Out(T)|$. Furthermore, as $H\cong G/K$ is quasiprimitive,
$X\overline{K}/\overline{K}$ is transitive.

Since $G/K$ has prime-power degree, so does $G$. If $G$ has type $\HS,
\HC, \TW, \SD$ or $\CD$, this is a contradiction as the degree of $G$
is divisible by $|T|$ as per Table~\ref{TblONS}.

We may thus assume that $G$ has type $\PA$. Here $|\Omega|=mx^k$ for some
$m\geqslant 1$, $x\geqslant 5$ and $k\geqslant 2$.   It follows from the
outline of the types in Section~\ref{sec:types} and Table~\ref{TblONS}
that $G\le Q\wr\sym_k$ where $Q$ is a quasiprimitive group of  degree $x$ and of $\AS$ type with $\soc(Q)\cong T$.  By Lemma \ref{lem:outT},  $|\Out(T)|<x$. This contradicts the fact that $mx^k$ divides $|X\overline{K}/\overline{K}|$ (as $X\overline{K}/\overline{K}$ is transitive of degree $mx^k$).

\subsubsection{$M\not\le K$}
In particular, $G$ has more than one minimal normal subgroup and thus
has type $\HS$ or $\HC$, and it has exactly two minimal normal
subgroups $N_1$ and $N_2$. We have $N_1\cong N_2\cong T^\ell$ for some
nonabelian simple group $T$ and thus $M=N_1\times N_2\cong T^{2\ell}$.

Let $\overline{\phantom{n}}$ be the natural epimorphism
$G\twoheadrightarrow G/K$. We may assume, without loss of generality,
that $N_1\le K$ and $N_2\not\le K$. In particular, $K\cap N_2=1$ and
hence $\overline{N_2}\cong N_2\cong T^\ell$ is a nontrivial minimal
normal subgroup of $\overline{G}$. Since $N_1\times N_2\leq G\lesssim
N_1 \rtimes\Aut(N_1)$, we have $\overline{N_2}\leq
\overline{G}\lesssim \Aut(T^\ell)$ and thus
$\soc(\overline{G})=\overline{N_2}$. Since $\overline{G}$ is
quasiprimitive, $\soc(\overline{G})$ is transitive. Moreover,
$\overline{G}$ has the same degree as $G$, namely
$|N_2|=|\soc(\overline{G})|$ and thus $\soc(\overline{G})$ is
regular. By Table~\ref{TblONS}, it follows that $\overline{G}$ has
type $\TW$ or type $\AS$ with a regular socle. If $\ell=1$ then $G$
has type $\HS$, while if $\ell\geq 2$ then $G$ has type $\HC$. These
cases appear in Theorem~\ref{theo:main}.

\section{Proof of Theorem~\ref{theo:main2}}\label{S1}

Let $L_-$ and $L_+$ be compatible quasiprimitive groups that are not isomorphic to each other. By the result of Knapp mentioned earlier (\cite{Knapp}*{Theorem~3.3}), we may assume that $L_+$ is a proper quotient of $L_-$. Since $L_+$ and $L_-$ have the same degree, we may apply Theorem~\ref{theo:main} to conclude that the pair $(L_-,L_+)$ has type one of $(\HS,\AS)$, $(\HC,\TW)$, $(\HA,\AS)$ or $(\HA,\PA)$. It thus only remains to exclude the last two possibilities. 

In view of obtaining a contradiction, we assume that $L_-$ has type $\HA$, and $L_+$ has type $\AS$ or $\PA$. By Theorem~\ref{theo:ProduceExamples}, there exists a finite group $H$ with two isomorphic subgroups $H_-$ and $H_+$ such that, for $\eps\in\{\pm\}$, the permutation group induced by the action of $H$ on the set of cosets of $H_\eps$ is permutation isomorphic to $L_\eps$. 

For $\eps\in\{\pm\}$, let $K_\eps$ be the core of $H_\eps$ in $H$ (so that $H/K_\eps\cong L_\eps$) and let $N_\eps/K_\eps$ be the unique minimal normal subgroup of $H/K_\eps$. Since $L_+$ has type $\AS$ or $\PA$, $N_+/K_+\cong T^k$ for some nonabelian simple group~$T$, and since $L_-$ has type $\HA$, $N_-/K_-\cong \C_p^d$ for some prime~$p$. Moreover, $H/N_-$ is an irreducible subgroup of $\GL(d,p)$ isomorphic to $H_-/K_-$. 

As $|L_+|<|L_-|$, we have $|K_+|>|K_-|$ therefore $K_-< K_-K_+$ and $K_-<N_-\le K_-K_+$. Using the notation of Definition~\ref{def:comp}, we have
$$
  [H_-]=[H_-/K_-][K_-/(K_-\cap K_+)][K_-\cap K_+]=[H_-/K_-][K_-K_+/K_+][K_-\cap K_+].
$$
Similarly, since $K_-<N_-\le K_-K_+$ and $[N_-/K_-]=[\C_p^d]$ we have
$$
  [H_+]=[H_+/K_+][K_-K_+/K_-][K_-\cap K_+]=[H_+/K_+][K_-K_+/N_-][\C_p^d][K_-\cap K_+].
$$
Since $H_-\cong H_+$, we have $[H_-]=[H_+]$. Canceling $[K_-\cap K_+]$ gives
\begin{equation}\label{supereq}
  [H_-/K_-][K_-K_+/K_+]=[H_+/K_+][K_-K_+/N_-][\C_p^d].
\end{equation}

If $K_+=K_-K_+$, then~\eqref{supereq} implies $[\C_p^d]\subseteq [H_-/K_-]$, which contradicts Proposition~\ref{compfactors}. Therefore  $K_+<K_-K_+$, and hence $K_+<N_+\le K_-K_+$. Now~\eqref{supereq} becomes:
$$
  [H_-/K_-][K_-K_+/N_+][T^k] =[H_+/K_+][K_-K_+/N_-][\C_p^d].
$$
Since $[H_-/K_-]=[H/N_-]=[H/K_-K_+][K_-K_+/N_-]$, we cancel $[K_-K_+/N_-]$ to get
$$
  [H/K_-K_+][K_-K_+/N_+][T^k] =[H_+/K_+][\C_p^d].
$$
Therefore $[T^k]\subseteq [H_+/K_+]$. This contradicts Lemma~\ref{CompFactorsPA} and concludes the proof.

\section{Examples for Theorems~\ref{theo:main2} and~\ref{theo:main}}\label{sec:examples}

We conclude by showing that the exceptional types that occur in Theorems~\ref{theo:main2} and~\ref{theo:main} cannot be avoided.

\begin{example}\label{ExampleHCTW}
Let $T$ be a nonabelian simple group, let $k\geq 1$ be an integer and let $N=T^k$. Let $H_-=\Inn(N)\rtimes S\leqslant\Aut(N)$, where $S\cong \sym_k$. Let $H=N\rtimes H_-$ and $H_+=N\rtimes S\leqslant H$. Since $T$ is centreless, $H_+\cong H_-$. For $\eps\in\{\pm\}$, let $L_\eps$ be the permutation group induced by the action of $H$ on the cosets of $H_\eps$. Then $L_-$ and $L_+$ are both quasiprimitive. If $k=1$, the type of $(L_-,L_+)$ is $(\HS,\AS)$ while if $k\geq 2$, the type is $(\HC,\TW)$. By Theorem~\ref{theo:ProduceExamples}, these permutation groups are compatible.
\end{example}

Example~\ref{ExampleHCTW} shows that the exceptions in Theorem~\ref{theo:main2} are necessary. By the result of Knapp mentioned in Section~\ref{sec:intro} (\cite{Knapp}*{Theorem~3.3}), any example for Theorem~\ref{theo:main2} will necessarily give an example for Theorem~\ref{theo:main}, and thus Example~\ref{ExampleHCTW} also provides examples of type $(\HS,\AS)$ and $(\HC,\TW)$ for Theorem~\ref{theo:main}. We now give examples of type $(\HA,\AS)$ and $(\HA,\PA)$.

\begin{example}\label{Ex:HAASPA} 
Observe that $\AGL(3,2)$ acts on $2^3=8$ affine points, and $\PSL(2,7)$ acts on $7+1=8$ projective points. Thus $\AGL(3,2)$ is a primitive group of type $\HA$ and degree $8$, while its quotient group $\GL(3,2)\cong\PSL(2,7)$ is a primitive group of type $\AS$, also of degree $8$. Let $k\geq 2$, let $P$ be a transitive group of degree $k$ and let $G=\AGL(3,2)\wr P$. In its natural action, this is a primitive permutation group of type $\HA$ and degree $8^k$. Let $N$ be the socle of $G$. Note that $G/N\cong \GL(3,2)\wr P\cong\PSL(2,7)\wr P$ and $G/N$ can be viewed as a primitive group of type $\PA$ and degree $8^k$.
\end{example}

It is reasonable to ask whether there are any other examples of type $(\HA,\AS)$ or $(\HA,\PA)$ for Theorem~\ref{theo:main}.

\end{document}